\documentclass[ejs]{imsart}

\RequirePackage[OT1]{fontenc}
\RequirePackage{amsthm,amsmath}
\RequirePackage[numbers]{natbib}
\RequirePackage[colorlinks,citecolor=blue,urlcolor=blue]{hyperref}

\pubyear{2006}
\volume{0}
\issue{0}
\firstpage{1}
\lastpage{8}

\usepackage{amsthm}
\usepackage{times,amssymb,amsmath,amsfonts,float,nicefrac,color,bbm,mathrsfs,caption,float}
\usepackage{algorithm,enumerate,multirow,caption,tikz,graphicx}
\usepackage{array}
\usepackage[mathscr]{eucal}
\usepackage{sidecap}
\usepackage{algpseudocode}
\usepackage{verbatim}
\usepackage{epstopdf}
\usepackage{textcomp}
\usetikzlibrary{shapes,arrows}
\usepackage{stmaryrd}
\usepackage{mathabx}
\usepackage{booktabs}
\usepackage[normalem]{ulem}
\allowdisplaybreaks

\newtheorem{theorem}{Theorem}[section]

\newtheorem{lemma}[theorem]{Lemma}
\newtheorem{proposition}[theorem]{Proposition}
\newtheorem{definition}[theorem]{Definition}

\newtheorem{cnstr}{Construction}

\newcounter{remark}[section]

   \newcounter{example}[section]

\newcommand{\RN}[1]{%
  \textup{\expandafter{\romannumeral#1}}%
}

\newcommand\remove[1]{}
\newcommand{\nc}{\newcommand}

\def\mathbi#1{{\textbf{\textit #1}}}
\nc\bfa{{\boldsymbol a}}\nc\bfA{{\boldsymbol A}}\nc\cA{{\mathscr A}}\nc\sA{{\mathscr A}}
\nc\bfb{{\boldsymbol b}}\nc\bfB{{\boldsymbol B}}\nc\cB{{\mathscr B}}\nc\sB{{\mathscr B}}
\nc\bfc{{\boldsymbol c}}\nc\bfC{{\boldsymbol C}}\nc\cC{{\mathscr C}}\nc\sC{{\mathscr C}}
\nc\bfd{{\boldsymbol d}}\nc\bfD{{\boldsymbol D}}\nc\cD{{\mathscr D}}
\nc\bfe{{\boldsymbol e}}\nc\bfE{{\boldsymbol E}}\nc\cE{{\mathscr E}}
\nc\bff{{\boldsymbol f}}\nc\bfF{{\boldsymbol F}}\nc\cF{{\mathscr F}}\nc\sF{{\mathscr F}}
\nc\bfg{{\boldsymbol g}}\nc\bfG{{\boldsymbol G}}\nc\cG{{\mathscr G}}
\nc\bfh{{\boldsymbol h}}\nc\bfH{{\boldsymbol H}}\nc\cH{{\mathscr H}}
\nc\bfi{{\boldsymbol i}}\nc\bfI{{\boldsymbol I}}\nc\cI{{\mathscr I}}\nc\sI{{\mathscr I}}
\nc\bfj{{\boldsymbol j}}\nc\bfJ{{\boldsymbol J}}\nc\cJ{{\mathscr J}}
\nc\bfk{{\boldsymbol k}}\nc\bfK{{\boldsymbol K}}\nc\cK{{\mathscr K}}
\nc\bfl{{\boldsymbol l}}\nc\bfL{{\boldsymbol L}}\nc\cL{{\mathscr L}}
\nc\bfm{{\boldsymbol m}}\nc\bfM{{\boldsymbol M}}\nc\cM{{\mathscr M}}
\nc\bfn{{\boldsymbol n}}\nc\bfN{{\boldsymbol N}}\nc\cN{{\mathcal N}}
\nc\bfo{{\boldsymbol o}}\nc\bfO{{\boldsymbol O}}\nc\cO{{\mathscr O}}
\nc\bfp{{\boldsymbol p}}\nc\bfP{{\boldsymbol P}}\nc\cP{{\mathscr P}}\nc\eP{{\EuScriptP}}\nc\fP{{\mathfrak P}}
\nc\bfq{{\boldsymbol q}}\nc\bfQ{{\boldsymbol Q}}\nc\cQ{{\mathscr Q}}
\nc\bfr{{\boldsymbol r}}\nc\bfR{{\boldsymbol R}}\nc\cR{{\mathscr R}}\nc\sR{{\mathscr R}}
\nc\bfs{{\boldsymbol s}}\nc\bfS{{\boldsymbol S}}\nc\cS{{\mathscr S}}
\nc\bft{{\boldsymbol t}}\nc\bfT{{\boldsymbol T}}\nc\cT{{\mathscr T}}
\nc\bfu{{\boldsymbol u}}\nc\bfU{{\boldsymbol U}}\nc\cU{{\mathscr U}}
\nc\bfv{{\boldsymbol v}}\nc\bfV{{\boldsymbol V}}\nc\cV{{\mathscr V}}\nc\sV{{\mathscr V}}
\nc\bfw{{\boldsymbol w}}\nc\bfW{{\boldsymbol W}}\nc\cW{{\mathscr W}}\nc\sW{{\mathscr W}}
\nc\bfx{{\boldsymbol x}}\nc\bfX{{\boldsymbol X}}\nc\cX{{\mathscr X}}
\nc\bfy{{\boldsymbol y}}\nc\bfY{{\boldsymbol Y}}\nc\cY{{\mathscr Y}}
\nc\bfz{{\boldsymbol z}}\nc\bfZ{{\boldsymbol Z}}\nc\cZ{{\mathscr Z}}

\DeclareMathOperator*{\argmin}{arg\!\min}

\DeclareMathOperator{\Vol}{Vol}

\DeclareMathOperator{\Bayes}{Bayes}

\DeclareMathOperator{\Unif}{Unif}

\begin{document}

\begin{frontmatter}
\title{Optimal locally private estimation under $\ell_p$ loss for $1\le p\le 2$}
\runtitle{Optimal locally private estimation under $\ell_p$ loss for $1\le p\le 2$}

\begin{aug}
\author{\fnms{Min} \snm{Ye}\thanksref{t1}\ead[label=e1]{yeemmi@gmail.com}}

\address{Department of Electrical Engineering\\
Princeton University \\
Princeton, NJ, 08544\\
\printead{e1}}

\author{\fnms{Alexander} \snm{Barg}\thanksref{t1}\ead[label=e2]{abarg@umd.edu}}

\address{Department of Electrical and Computer Engineering\\
and Institute for Systems Research\\
 University of Maryland \\
College Park, MD 20742\\
\printead{e2}}

\thankstext{t1}{Research partially supported by NSF grants CCF1422955 and CCF1618603.}
\runauthor{M. Ye and A. Barg}

\affiliation{Princeton University and University of Maryland, College Park}

\end{aug}

\begin{abstract}
We consider the minimax estimation problem of a discrete distribution with support size $k$ under locally differential privacy constraints. A privatization scheme is applied to each raw sample independently, and we need to estimate the distribution of the raw samples from the privatized samples. A positive number $\epsilon$ measures the privacy level of a privatization scheme. 

In our previous work ({\em IEEE Trans. Inform. Theory}, 2018), we proposed a family of new privatization schemes and the corresponding estimator. We also proved that our scheme and estimator are order optimal in the regime $e^{\epsilon} \ll k$ under both $\ell_2^2$ (mean square) and $\ell_1$ loss. 
In this paper, we sharpen this result by showing asymptotic optimality of the proposed scheme under the $\ell_p^p$ loss for all $1\le p\le 2.$ More precisely, we show that for any $p\in[1,2]$ and any $k$ and $\epsilon,$ the ratio between the worst-case $\ell_p^p$ estimation loss of our scheme and the optimal value approaches $1$ as the number of samples tends to infinity. The lower bound on the minimax risk of private estimation that we establish as a part of the proof is valid for any loss function $\ell_p^p, p\ge 1.$
\end{abstract}



\end{frontmatter}

\section{Introduction} 

This paper continues our work \cite{Ye17}. The context of the problem that we consider is
related to a major challenge in the statistical analysis of user data, namely, the conflict between learning accurate 
statistics and protecting sensitive information about the individuals. 
As in \cite{Ye17}, we rely on a particular formalization of user privacy called {\em differential privacy}, introduced in \cite{Dwork06, Dwork08}.
Generally speaking, differential privacy requires that the adversary not be able to reliably infer an individual's data from public statistics even with access to all the other users' data.
The concept of differential privacy has been developed in two different contexts: the {\em global privacy}
context (for instance, when institutions release statistics related to groups of people) \cite{Ghosh12}, and the {\em local privacy} context when individuals disclose their personal data \cite{Duchi13}.

In this paper, we consider the minimax estimation problem of a discrete distribution with support size $k$ under locally differential privacy.
This problem has been studied in the non-private setting \cite{Kamath15, Lehmann06}, where we can learn the distribution from the raw samples.
In the private setting, we need to estimate the distribution of raw samples from the privatized samples which are generated independently from the raw samples according to a conditional distribution  $\mathbi{Q}$ (also called a {\em privatization scheme}).
Given a privacy parameter $\epsilon>0,$
we say that $\mathbi{Q}$ is $\epsilon$-locally differentially private if the probabilities of the same output conditional on different inputs differ by a factor of at most $e^{\epsilon}.$ Clearly, smaller $\epsilon$ means that it is more difficult to infer the original data from the privatized samples, and thus leads to higher privacy.
For a given $\epsilon,$ our objective is to find the optimal $\epsilon$-private scheme that minimizes the expected estimation loss for the worst-case distribution.
In this paper, we are mainly
concerned with the scenario where we have a large number of
samples, which captures the modern trend toward ``big data" analytics.

\subsection{Existing results}\label{sec:existing} The following two privatization schemes are the most well-known in the literature: the $k$-ary Randomized Aggregatable Privacy-Preserving Ordinal Response ($k$-RAPPOR) scheme \cite{Duchi13a, Erlingsson14}, and the $k$-ary Randomized Response ($k$-RR) scheme
\cite{Warner65,Kairouz14}.
The $k$-RAPPOR scheme is order optimal in the high privacy regime where $\epsilon$ is very close to $0,$
and the $k$-RR scheme is order optimal in the low privacy regime where $e^{\epsilon} \approx k$ \cite{Kairouz16}. 
Very recently, a family of privatization schemes and the corresponding estimators were proposed independently by Wang et al. \cite{Wang16} and the present authors \cite{Ye17}. In \cite{Ye17}, we further showed that
under both $\ell_2^2$ (mean square) and $\ell_1$ loss, these privatization schemes and the corresponding estimators are order-optimal in the medium to high privacy regimes when $e^{\epsilon} \ll k.$ 
Subsequent to our work, \cite{Acharya18} proposed
another privatization scheme and proved that it is order optimal in all regimes for $\ell_1$ loss. At the same time, prior to
this paper, no schemes were shown to be asymptotically optimal in the literature.

Duchi et al.~\cite{Duchi16} gave an order-optimal lower bound on the minimax private estimation loss for the high privacy regime where $\epsilon$ is very close to $0$. In \cite{Ye17}, we proved a stronger lower bound which is order-optimal in the whole region $e^{\epsilon} \ll k$. This lower bound implies that the schemes and the estimators proposed in \cite{Wang16,Ye17} are order optimal in this regime.   
Here order-optimal means that the ratio between the true value and the lower bound is upper bounded by a constant (larger than 1) when $n$ and $k/e^{\epsilon}$ both become large enough. 

\subsection{Our contributions}
In this paper, we study the private estimation problem under the $\ell_p^p$ loss for $1\le p\le 2$, which in particular includes the widely used $\ell_1$ and $\ell_2^2$ loss. We prove an asymptotically tight lower bound on the $\ell_p^p$ loss of  the minimax private estimation for all values of $k,\epsilon$ and $1\le p\le 2$.  This improves upon the lower bounds in \cite{Ye17} and \cite{Duchi16} for the following three reasons: First, although the lower bounds in \cite{Ye17} and \cite{Duchi16} are order-optimal, they differ from the true value by a factor of several hundred. In practice, an improvement of several percentage points is already considered as a substantial advance (see for instance,~\cite{Kairouz16}), so tighter bounds are of interest. Second, the bounds in \cite{Ye17} and \cite{Duchi16} only hold for certain regions of $k$ and $\epsilon$ while the lower bound in this paper holds for all values of $k$ and $\epsilon$.
Finally, previous results were limited to $\ell_1$ and $\ell_2^2$ loss functions while the results in this paper hold for all $\ell_p^p$ loss functions, where $1\le p\le 2$.

Furthermore, as an immediate consequence of our lower bound, we show that the schemes and the estimators proposed in \cite{Wang16,Ye17} are universally optimal under the $\ell_p^p$ loss for all $1\le p\le 2$ in the sense that the ratio between the lower bound and the worst-case estimation loss of these schemes and estimators goes to $1$ when $n$ goes to infinity.

In this paper we both generalize the results, and shorten the proofs in the preprint \cite{Ye17tight} which addressed only the case of 
mean square loss.

\subsection{Related work}
While in this paper we consider only the sample complexity, 
a recent work by Acharya et al. \cite{Acharya18} took communication complexity into consideration and proposed a new privatization scheme with reduced communication complexity while maintaining the optimal order of sample complexity for the $\ell_1$ loss function.
Apart from the $\ell_p$ loss measures considered in this paper, significant attention in 
the literature was devoted to the $\ell_{\infty}$ estimation of a discrete distribution  (also called the heavy hitters problem) under local differential privacy \cite{Mishra06,Hsu12,Bassily15}. 
Although we only consider the case where the same privatization scheme is applied to each raw sample in this paper, one can also construct privatization schemes that depend on the values of previously observed privatized samples. Such interactive privatization schemes are important for online and sequential procedures in private learning \cite{Smith11,Thakurta13,Duchi16}.  A recent work \cite{Acharya18a} addresses the private estimation problem of distributional properties when the support size $k$ is not known to the estimator.
Other estimation-related problems that were
studied under local differential privacy constraints include the problem of testing identity and closeness of discrete distributions \cite{Acharya17} and hypothesis testing \cite{Gaboardi17}.

\subsection{Organization of the paper}
In Section~\ref{Sect:pre}, we formulate the problem and give a more detailed review of the existing results.
Section~\ref{Sect:ovr} is devoted to an overview of the main results of this paper.
The proofs of the main results are given in Sections~\ref{Sect:LAN}-\ref{Sect:pac}.

\section{Problem formulation and existing results}\label{Sect:pre}
\textbf{Notation:}
Let $\cX=\{1,2,\dots,k\}$ be the source alphabet and let $\mathbi{p}=(p_1,p_2,\dots,p_k)$ be a probability distribution on $\cX.$
Denote by $\Delta_k=\{\mathbi{p}\in \mathbb{R}^k: p_i\ge 0 \text{~for~} i=1,2,\dots,k, \sum_{i=1}^k p_i=1\}$ the $k$-dimensional probability simplex. Let $X$ be a random variable (RV) that takes values on $\cX$ according to $\mathbi{p}$, so that $p_i=P(X=i).$ Denote by $X^n=(X^{(1)},X^{(2)},\dots,X^{(n)})$ the vector formed of $n$ independent copies of the RV $X.$

\subsection{Problem formulation}
In the classical (non-private) distribution estimation problem, we are given direct access to i.i.d. samples
$\{X^{(i)}\}_{i=1}^n$ drawn according to some unknown distribution $\mathbi{p}\in \Delta_k.$ Our goal is to estimate $\mathbi{p}$ based on the samples \cite{Lehmann06}. We define an estimator $\hat{\mathbi{p}}$ as a function
$\hat{\mathbi{p}}:\cX^n \to \mathbb{R}^k,$ and assess its quality in terms of the worst-case risk (expected loss)
$$
\sup_{\mathbi{p}\in \Delta_k}  \underset{X^n\sim \mathbi{p}^n}{\mathbb{E}} \ell(\hat{\mathbi{p}}(X^n), \mathbi{p}),
$$
where $\ell$ is some loss function. The minimax risk is defined as the solution of the following saddlepoint problem:
$$
r_{k,n}^{\ell}:= \inf_{\hat{\mathbi{p}}} \sup_{\mathbi{p}\in \Delta_k} 
\underset{X^n\sim \mathbi{p}^n}{\mathbb{E}} \ell(\hat{\mathbi{p}}(X^n), \mathbi{p}).
$$

In the private distribution estimation problem, we can no longer access the raw samples $\{X^{(i)}\}_{i=1}^n.$ Instead, we estimate the distribution $\mathbi{p}$ from the privatized samples $\{Y^{(i)}\}_{i=1}^n,$ obtained by applying a privatization mechanism $\mathbi{Q}$ independently to each raw sample $X^{(i)}.$ A {\em privatization mechanism} (also called privatization scheme) $\mathbi{Q}:\cX\to\cY$ is simply a conditional distribution $\mathbi{Q}_{Y|X}.$ The
privatized samples $Y^{(i)}$ take values in a set $\cY$ (the ``output alphabet'') that does not have to be the same as $\cX.$

The quantities $\{Y^{(i)}\}_{i=1}^n$ are i.i.d. samples drawn according to the marginal distribution $\mathbi{m}$ given by
\begin{equation}\label{eq:defm}
\mathbi{m}(S)=\sum_{i=1}^k \mathbi{Q}(S|i)p_i
\end{equation}
 for any $S\in \sigma(\cY),$ where $\sigma(\cY)$ denotes an appropriate $\sigma$-algebra on $\cY.$
In accordance with this setting, the estimator $\hat{\mathbi{p}}$ is a measurable function $\hat{\mathbi{p}}:\cY^n\to \mathbb{R}^k.$
We assess the quality of the privatization scheme $\mathbi{Q}$ and the corresponding estimator $\hat{\mathbi{p}}$ by the worst-case risk 
$$
r_{k,n}^{\ell} (\mathbi{Q}, \hat{\mathbi{p}}) := \sup_{\mathbi{p}\in \Delta_k} 
\underset{Y^n\sim \mathbi{m}^n}{\mathbb{E}} \ell(\hat{\mathbi{p}}(Y^n), \mathbi{p}),
$$
where $\mathbi{m}^n$ is the $n$-fold product distribution and $\mathbi m$ is given by \eqref{eq:defm}.
Define the {\em minimax risk} of the privatization scheme $\mathbi{Q}$ as
  \begin{equation}\label{eq:riskQ}
r_{k,n}^{\ell} (\mathbi{Q}):= \inf_{\hat{\mathbi{p}}} r_{k,n}^{\ell} (\mathbi{Q}, \hat{\mathbi{p}}).
   \end{equation}
\begin{definition}
For a given $\epsilon>0,$
a privatization mechanism $\mathbi{Q}:\cX\to\cY$ is said to be {\em $\epsilon$-locally differentially private} if 
for all $x,x'\in\cX$
\begin{equation}\label{eq:defep}
\sup_{S\in\sigma(\cY)} \log \frac{\mathbi{Q}(Y\in S|X=x)}{\mathbi{Q}(Y\in S|X=x')} \le {\epsilon}.
\end{equation}
\end{definition}

Denote by $\cD_{\epsilon}$ the set of all $\epsilon$-locally differentially private mechanisms. Given a privacy level $\epsilon$ and a loss function $\ell$, we seek to find the optimal $\mathbi{Q}\in\cD_{\epsilon}$ with the smallest possible minimax risk $r_{k,n}^{\ell} (\mathbi{Q})$ among all the 
$\epsilon$-locally differentially private mechanisms. 
  As already mentioned,  in this paper we will consider\footnote{The standard notation for the loss function should be $\ell_p^p$, as we used in the Introduction. However, in order to avoid confusion with the notation for probability distribution, we will use $\ell_u^u$ from now on.} $\ell=\ell_u^u$ for $1\le u\le 2$, where for
$x=(x_1,x_2,\dots,x_k)\in\mathbb{R}^k$
$$
\ell_u^u (x):= \sum_{i=1}^k |x_i|^u .
$$
It is easy to see that for any valid privatization scheme $\mathbi{Q}$, the order of its $\ell_u^u$ minimax estimation risk is $\Theta(n^{-u/2})$, and $\lim_{n\to\infty} r_{k,n}^{\ell_u^u} (\mathbi{Q}) n^{u/2}$ is the coefficient of the dominant term, which measures the performance of $\mathbi{Q}$ when $n$ is large.
  
\vspace*{.1in}
\noindent{\bf Main Problem:} \emph{Suppose that the cardinality $k$ of the source alphabet is known to the estimator.
For a given privacy level $\epsilon$, we would like to find the optimal (smallest possible) value of $\lim_{n\to\infty} r_{k,n}^{\ell_u^u} (\mathbi{Q}) n^{u/2}$ among all $\mathbi{Q}\in\cD_{\epsilon}$ and to construct a privatization mechanism and a corresponding  estimator to achieve this optimal value.}
  
\vspace*{.1in}  It is this problem that we address---and resolve---in this paper. Specifically, we prove a lower bound
on $\lim_{n\to\infty} r_{k,n}^{\ell_u^u} (\mathbi{Q}) n^{u/2}$ for $\mathbi{Q}\in\cD_{\epsilon}$, which implies that the mechanism and the corresponding estimator
proposed in \cite{Ye17} are universally optimal for all loss functions $\ell_u^u, 1\le u\le 2$.

\subsection{Previous results}
In this section we briefly review known results that are relevant to our problem. In Sect.~\ref{sec:existing} we mentioned
several papers that have considered it, viz., \cite{Warner65,Duchi13a,Erlingsson14,Kairouz14,Kairouz16,Wang16,Duchi16,Acharya18}. 
In this section we focus on the results of \cite{Ye17} because they are stated in the form convenient for our presentation.

Let $\cD_{\epsilon,F}$ be the set of $\epsilon$-locally differentially private schemes with finite output alphabet. 
Let
  \begin{equation}\label{eq:DES}
\cD_{\epsilon,E}=\biggl\{ \mathbi{Q}\in\cD_{\epsilon,F}: 
\frac{\mathbi{Q}(y|x)}{\min_{x'\in\cX}\mathbi{Q}(y|x') } \in \{1,e^{\epsilon}\}
\text{~for all~} x\in\cX \text{~and all~} y\in\cY \biggr\}.
  \end{equation}
 In \cite[Theorem 13]{Ye17}, we have shown that
\begin{equation}\label{eq:red}
r_{k,n}^{\ell_u^u} (\mathbi{Q}) \ge \inf_{\mathbi{Q}'\in\cD_{\epsilon,E}} r_{k,n}^{\ell_u^u} (\mathbi{Q}')    \text{~~~for all~} \mathbi{Q}\in\cD_{\epsilon}.
\end{equation}
As a result, below we limit ourselves to schemes $\mathbi{Q}\in\cD_{\epsilon,E}$ in this paper. 
For such schemes, since the output alphabet is finite, we can write the marginal distribution $\mathbi{m}$  in \eqref{eq:defm} 
as a vector $\mathbi{m}=(\sum_{j=1}^k p_j \mathbi{Q}(y|j), y\in\cY).$ We will also use the shorthand notation
$\mathbi{m}=\mathbi{p}\mathbi{Q}$ to denote this vector.

In \cite{Ye17}, we introduced a family of privatization schemes
which are parameterized by the integer $d\in\{1,2,\dots,k-1\}.$ Given $k$ and $d,$ let the output alphabet be 
$\cY_{k,d}=\{y\in \{0,1\}^k: \sum_{i=1}^k y_i=d\},$ so $|\cY_{k,d}|=\binom{k}{d}.$ 

\begin{definition} [\cite{Ye17}] Consider the following privatization scheme:
   \begin{equation}\label{eq:defQ}
\mathbi{Q}_{k,\epsilon,d}(y|i)=\frac{e^\epsilon y_i+(1-y_i)}{\binom{k-1}{d-1}e^{\epsilon}+\binom{k-1}{d}} 
    \end{equation}
for all $y\in\cY_{k,d}$ and all $i\in\cX.$
The corresponding empirical estimator of $\mathbi{p}$ under $\mathbi{Q}_{k,\epsilon,d}$ is defined as follows: For $y^n=(y^{(1)}, y^{(2)}, \dots, y^{(n)})\in \cY_{k,d}^n$,
\begin{equation}\label{eq:emp}
\hat{p_i}(y^n)=\Big(\frac{(k-1)e^{\epsilon}+\frac{(k-1)(k-d)}{d}}{(k-d)(e^{\epsilon}-1)}\Big)
\frac{t_i(y^n)}{n}
-\frac{(d-1)e^{\epsilon}+k-d}{(k-d)(e^{\epsilon}-1)},  \quad i\in[k]
\end{equation}
where $t_i(y^n)=\sum_{j=1}^n y_i^{(j)}$ is the number of privatized samples whose $i$-th coordinate is $1$.
\end{definition}

Some papers \cite{Acharya18} call $\mathbi{Q}_{k,\epsilon,d}$ the {\em Subset Selection} mechanism.
It is easy to verify that $\mathbi{Q}_{k,\epsilon,d}$ is $\epsilon$-locally differentially private.
The worst-case estimation loss under $\mathbi{Q}_{k,\epsilon,d}$ and the empirical estimator is calculated in the following proposition.

\begin{proposition}\label{prop:risks}{\rm\cite[Prop.~4-5]{Ye17}} Let
 $\mathbi{Q}=\mathbi{Q}_{k,\epsilon,d}$ and suppose that the empirical estimator $\hat{\mathbi{p}}$ is given by \eqref{eq:emp}.
Let $\mathbi{m}=\mathbi{p}\mathbi{Q}_{k,\epsilon,d}.$
The estimation loss $\underset{Y^n\sim \mathbi{m}^n}{\mathbb{E}} \ell_2^2(\hat{\mathbi{p}}(Y^n),\mathbi{p})$ is maximized for the uniform distribution $\mathbi{p}_U=(1/k,1/k,\dots,1/k)$, and 
  \begin{equation}\label{eq:rd}
r_{k,n}^{\ell_2^2} (\mathbi{Q}_{k,\epsilon,d}, \hat{\mathbi{p}})  =
\underset{Y^n\sim \mathbi{m}_U^n}{\mathbb{E}} \ell_2^2(\hat{\mathbi{p}}(Y^n),\mathbi{p}_U)=
\frac{(k-1)^2}{nk(e^{\epsilon}-1)^2} \frac{(d e^{\epsilon} + k-d)^2 }{d(k-d)},
  \end{equation}
  where $\mathbi{m}_U=\mathbi{p}_U \mathbi{Q}_{k,\epsilon,d}.$
\end{proposition}

It is clear that the smallest value of the risk $\mathbi r$ is obtained by optimizing on $d$ in \eqref{eq:rd}. Namely, given $k$ and $\epsilon$, let
\begin{equation}\label{eq:dstar}
d^\ast = d^\ast(k,\epsilon)
:= \argmin_{1\le d \le k-1}  \frac{(d e^{\epsilon} + k-d)^2 }{d(k-d)},
\end{equation}
where the ties are resolved arbitrarily. We find that $d^\ast$ takes one the following two values:
$$
d^\ast=\lceil k/(e^{\epsilon}+1) \rceil \text{ or } \lfloor k/(e^{\epsilon}+1)\rfloor.
$$
Therefore, when $k/(e^{\epsilon}+1) \le 1$, $d^\ast=1,$ and when $k/(e^{\epsilon}+1) > 1$, the value of $d^\ast$ can be determined
by simple comparison.

As a consequence of Prop.~\ref{prop:risks} we find that
  \begin{equation*}
r_{k,n}^{\ell_2^2} (\mathbi{Q}_{k,\epsilon,d^\ast}, \hat{\mathbi{p}}) = 
\min_{1\le d \le k-1} r_{k,n}^{\ell_2^2} (\mathbi{Q}_{k,\epsilon,d}, \hat{\mathbi{p}}).
  \end{equation*}   
While in \cite{Ye17} we proved the above results for the mean-square loss (and a similar claim for $\ell=\ell_1$), in this paper we show that they apply more universally. Namely, let
 \begin{equation}\label{eq:Mke}
     M(k,\epsilon):= \frac{(k-1)^2}{k^2(e^{\epsilon}-1)^2}\frac{(d^\ast e^{\epsilon} + k-d^\ast)^2 }{d^\ast(k-d^\ast)}.
  \end{equation}
and note that $r_{k,n}^{\ell_2^2} (\mathbi{Q}_{k,\epsilon,d^\ast}, \hat{\mathbi{p}})
=\frac{k}{n}M(k,\epsilon).$ In this paper we show that the quantity $M(k,\epsilon)$ bounds below the main term of the minimax risk 
for all loss functions $\ell_u^u,u\ge 1.$

\section{Main result of the paper}\label{Sect:ovr}
Our main result is that the scheme
$\mathbi{Q}_{k,\epsilon,d^\ast}$ and the empirical estimator $\hat{\mathbi{p}}$ defined by \eqref{eq:emp} are universally optimal for all loss functions $\ell_u^u, 1\le u\le 2$. Namely, the following is true.
\begin{theorem} Let $k=|\cX|,$ let $\epsilon>0,1\le u\le 2$. Then
$$
\lim_{n\to\infty} \frac{r_{k,n}^{\ell_u^u} (\mathbi{Q})} {r_{k,n}^{\ell_u^u} (\mathbi{Q}_{k,\epsilon,d^\ast}, \hat{\mathbi{p}})} \ge 1
\text{~~~for all~} \mathbi{Q}\in\cD_{\epsilon}.
$$
\end{theorem}

This theorem is a consequence of two results which we state next.

Let $X\sim \cN(0,1)$ and define the constant 
 $$
 C_u:=E|X|^u=2^{u/2}\Gamma((u+1)/2)/\sqrt\pi  \quad \text{~for~} u>0.
$$

\begin{theorem}\label{Thm:Main}
For any $\epsilon>0,$ any $u\ge 1$, and any mechanism $\mathbi{Q}\in \cD_{\epsilon}$
\begin{equation} \label{eq:cnlb}
\lim_{n\to\infty} r_{k,n}^{\ell_u^u} (\mathbi{Q}) n^{u/2} \ge
k C_u M(k,\epsilon)^{u/2}.
\end{equation}
\end{theorem}
Note that this lower bound holds for any loss function $\ell_u^u, u\ge 1$.  The proof of this theorem is given in Section~\ref{Sect:LAN}.

\begin{theorem} \label{thm:ac} Consider the privatization 
scheme $\mathbi{Q}=\mathbi{Q}_{k,\epsilon,d^\ast}$ and let $\hat{\mathbi{p}}$ be the empirical estimator given by \eqref{eq:emp}.
For every $k$ and $\epsilon$ and every $0< u\le 2$,
$$
r_{k,n}^{\ell_u^u} (\mathbi{Q}_{k,\epsilon,d^\ast}, \hat{\mathbi{p}}) = 
\frac{k}{n^{u/2}} C_u M(k,\epsilon)^{u/2} + o(n^{-u/2}).
$$
\end{theorem}
The proof of this theorem is given in Section~\ref{Sect:pac}.
Note that, unlike Theorem \ref{Thm:Main}, the claim that we make here allows the values of $u\in(0,1)$. 
The special cases of Theorem \ref{thm:ac} for $u=1$ and $u=2$ were addressed in our previous paper \cite{Ye17}, see in particular
Theorem 10.

The crux of our argument is in the proof of Theorem~\ref{Thm:Main}, where we reduce the estimation problem in the $k$-dimensional space 
to a one-dimensional problem. Generally, it is well known that the local minimax risk can be 
calculated from the inverse of the Fisher information matrix. However, it is difficult to obtain the exact expression of the inverse of 
a large-size matrix, and without it, the path to the desired estimates is not so clear. To work around this complication,
we view a ball in a high-dimensional space as a union of parallel line segments with a certain direction $\mathbi{v}_i$. We first consider the estimation problem on each line segment individually. Since this is a one-dimensional problem, 
its minimax rate can be easily calculated from the Fisher information of the corresponding parameter. For the estimation of each component $p_i$ of the probability distribution, we choose a suitable direction vector  $\mathbi{v}_i$. In this way, we reduce the original $k$-dimensional estimation problem to $k$ one-dimensional 
estimation problems and then rely on the additivity of the loss function for the final result.

\section{Proof of Theorem~\ref{Thm:Main}}\label{Sect:LAN}
\subsection{Bayes estimation loss}
In light of \eqref{eq:red}, to prove Theorem~\ref{Thm:Main}, it suffices to show that for every $u\ge 1$,
\begin{equation} \label{eq:dp}
\lim_{n\to\infty} r_{k,n}^{\ell_u^u} (\mathbi{Q}) n^{u/2} \ge
k C_u M(k,\epsilon)^{u/2}
 \text{~~~for all~} \mathbi{Q}\in\cD_{\epsilon,E}. 
\end{equation}

Since the worst-case estimation loss is always lower bounded by the average estimation loss, the minimax risk $r_{k,n}^{\ell_u^u} (\mathbi{Q})$ can be bounded below by the Bayes estimation loss. 
More specifically, we assume that $\mathbi{p}:=\{p_1,p_2,\dots,p_k\}$ is drawn uniformly from 
\begin{equation}\label{eq:npu}
\cP:=\Big\{ \mathbi{p}\in\Delta_k:\|\mathbi{p}-\mathbi{p}_U\|_2 \le \frac{D}{\sqrt n}  \Big\},
\end{equation}
 where $D\gg 1$ is a constant. Let 
$\mathbi{P}=(P_1,P_2,\dots,P_k)$ denote the random vector that corresponds to $\mathbi{p}$.
For a given privatization scheme $\mathbi{Q}$ and the corresponding estimator $\hat{\mathbi{p}}:=(\hat{p}_1,\hat{p}_2,\dots,\hat{p}_k)$, the $\ell_u^u$ Bayes estimation loss is defined as
\begin{align*}
r_{\Bayes}^{\ell_u^u} (\mathbi{Q}, \hat{\mathbi{p}}) := &
\underset{\mathbi{P}\sim \Unif(\cP)}{\mathbb{E}}
\Big[ \underset{Y^n\sim (\mathbi{P}\mathbi{Q})^n}{\mathbb{E}} \ell_u^u(\hat{\mathbi{p}}(Y^n), \mathbi{P}) \Big] \\
= & \sum_{i=1}^k \left( \underset{\mathbi{P}\sim \Unif(\cP)}{\mathbb{E}}
\Big[ \underset{Y^n\sim (\mathbi{P}\mathbi{Q})^n}{\mathbb{E}}  |\hat{p}_i(Y^n)-P_i|^u \Big] \right),
\end{align*}
and the optimal Bayes estimation loss for $\mathbi{Q}$ is
$$
r_{\Bayes}^{\ell_u^u} (\mathbi{Q}) := 
\inf_{\hat{\mathbi{p}}}  r_{\Bayes}^{\ell_u^u} (\mathbi{Q}, \hat{\mathbi{p}}).
$$
We further define component-wise Bayes estimation loss for $\mathbi{Q}$ and $\hat{\mathbi{p}}$
$$
r_{i,\Bayes}^{\ell_u^u} (\mathbi{Q}, \hat{p}_i) :=
\underset{\mathbi{P}\sim \Unif(\cP)}{\mathbb{E}}
\Big[ \underset{Y^n\sim (\mathbi{P}\mathbi{Q})^n}{\mathbb{E}}  |\hat{p}_i(Y^n)-P_i|^u \Big],
\quad i\in[k],
$$
and the optimal component-wise Bayes estimation loss for $\mathbi{Q}$
$$
r_{i,\Bayes}^{\ell_u^u} (\mathbi{Q}):= \inf_{\hat{p}_i}
r_{i,\Bayes}^{\ell_u^u} (\mathbi{Q}, \hat{p}_i),  \quad i\in[k].
$$
Therefore,
$$
r_{\Bayes}^{\ell_u^u} (\mathbi{Q}, \hat{\mathbi{p}})
= \sum_{i=1}^k  r_{i,\Bayes}^{\ell_u^u} (\mathbi{Q}, \hat{p}_i),
\quad
r_{\Bayes}^{\ell_u^u} (\mathbi{Q}) =  \sum_{i=1}^k
r_{i,\Bayes}^{\ell_u^u} (\mathbi{Q}).
$$
As mentioned above, 
$$
r_{k,n}^{\ell_u^u} (\mathbi{Q}) \ge r_{\Bayes}^{\ell_u^u} (\mathbi{Q}) =  \sum_{i=1}^k
r_{i,\Bayes}^{\ell_u^u} (\mathbi{Q}).
$$
We will prove \eqref{eq:dp} by showing that
\begin{equation}\label{eq:mh}
\sum_{i=1}^k r_{i,\Bayes}^{\ell_u^u} (\mathbi{Q}) \ge
\frac{k}{n^{u/2}} C_u M(k,\epsilon)^{u/2} - o(n^{-u/2})
 \text{~~~for all~} \mathbi{Q}\in\cD_{\epsilon,E}.
\end{equation}

\subsection{Lower bound on one-dimensional Bayes estimation loss}
Below we will prove a lower bound on $r_{i,\Bayes}^{\ell_u^u} (\mathbi{Q})$. To this end, in this section we consider a one-dimensional Bayes estimation problem.
Define the following vectors:
\begin{equation}\label{eq:vv}
\mathbi{v}_i:=\Big(-\frac1{k-1},\dots,-\frac1{k-1},1,-\frac1{k-1},\dots,-\frac1{k-1}\Big), \quad i\in[k].
\end{equation}
where the 1 is in the $i$th position and all the other coordinates are $-\frac1{k-1}$.
 Let $\mathbi{p}^\ast:=(p_1^\ast,p_2^\ast,\dots,p_k^\ast)\in\Delta_k$ be a probability distribution and let
$S_i(\mathbi{p}^\ast)$ be a line segment with midpoint $\mathbi{p}^\ast$ and direction vector $\mathbi{v}_i$:
\begin{equation}\label{eq:rgh}
 S_i(\mathbi{p}^\ast):= \Big\{ \mathbi{p}^\ast+s \mathbi{v}_i:|s|\le\frac{D'}{\sqrt{n}} \Big\}, \quad i\in[k],
\end{equation}
where $D'\gg 1$ is a constant. Let $\mathbi{p}=(p_1,\dots,p_k)$ be a PMF in the segment $S_i(\mathbi{p}^\ast)$. Given the value $p_i,$ we can find all the other components of $\mathbi{p}$ as follows:
\begin{equation}\label{eq:cpi}
p_v= p_v^\ast-  \frac{1}{k-1}(p_i-p_i^\ast) \text{~~~~for all~} v\neq i.
\end{equation}
Assume that $\mathbi{p}=(p_1,p_2,\dots,p_k)$ is drawn uniformly from $S_i(\mathbi{p}^\ast)$, and we consider the Bayes estimation of $p_i$ from the privatized samples $Y^n$ obtained from applying $\mathbi{Q}$ to the raw samples. More precisely, for an estimator $\hat{p}_i$, we define its Bayes estimation loss
$$
r_{i,S_i(\mathbi{p}^\ast)}^{\ell_u^u} (\mathbi{Q}, \hat{p}_i) :=
\underset{\mathbi{P}\sim \Unif(S_i(\mathbi{p}^\ast))}{\mathbb{E}}
\Big[ \underset{Y^n\sim (\mathbi{P}\mathbi{Q})^n}{\mathbb{E}}  |\hat{p}_i(Y^n)-P_i|^u \Big],
\quad i\in[k],
$$
then the optimal estimation loss is
$$
r_{i,S_i(\mathbi{p}^\ast)}^{\ell_u^u} (\mathbi{Q}) := \inf_{\hat{p}_i} r_{i,S_i(\mathbi{p}^\ast)}^{\ell_u^u} (\mathbi{Q}, \hat{p}_i), \quad i\in[k].
$$
Our approach to obtain the lower bound on this Bayes estimation loss relies on a classical method in asymptotic statistics, namely, local asymptotic normality (LAN) of the posterior distribution \cite{LeCam12,Ibrag81,Hajek72,LCYang12}. More specifically, let $P_i$ be the random variable corresponding to $p_i$.
According to the well-known results in the LAN literature (see for instance \cite[Chapter 2, Theorem 1.1]{Ibrag81} and \cite[Chapter 6]{LCYang12}), when the constant $D'$ is large enough, the conditional distribution of $P_i$ given $Y^n=y^n$ is approximately a Gaussian distribution with variance $(I(p_i^\ast))^{-1}$ for 
almost all\footnote{More precisely, for any $\epsilon_1,\epsilon_2>0$ there is $N$ such that for any $n>N$ there is a subset $E
\subseteq \cY^n$ such that (1) $\mathbb{P}(E)>1-\epsilon_1,$ and (2) for all $y^n\in E$ the relative difference between the pdf of conditional distribution of $P_i$ given $Y^n=y^n$ and the Gaussian pdf is at most $\epsilon_2$.}
$y^n\in\cY^n$ as $n$ goes to infinity, where $I(\cdot)$ is the Fisher information of the parameter $p_i$.
Before we calculate the value of $I(p_i^\ast)$, let us recall a simple fact about Gaussian distribution:
Suppose that $X$ is a Gaussian random variable, then one can easily verify\footnote{
Let $\phi(x)$ be the pdf of $X$ and note that $\phi(x)=\phi(2\mathbb{E}X-x)$ for all real $x.$ By convexity of $|\cdot|^u, u\ge 1$ we have
  $|x-\mathbb{E}X|^u\le (1/2)(|a-x|^u+|2\mathbb{E}X-x-a|^u) 
$ for all $a$. Integrating against $\phi(x)$ and using the symmetry condition, we obtain that
$\mathbb{E}|X-\mathbb{E}X|^u\le \mathbb{E}|X-a|^u$ for all $u\ge 1, a\in\mathbb{R}.$
} that for any $u\ge 1,$
\begin{equation}\label{eq:ji}
\mathbb{E}X=\argmin_{a} \mathbb{E}|X-a|^u.
\end{equation}
Therefore, the estimator $\hat{p}_i(y^n)=\mathbb{E}(P_i|Y^n=y^n)$ is asymptotically optimal for this Bayes estimation problem under the $\ell_u^u$ loss function for all $u\ge 1$. Since the variance of $P_i$ given $Y^n=y^n$ is $(I(p_i^\ast))^{-1}$ for almost all $y^n\in\cY^n$, the Bayes estimation loss of this asymptotically optimal estimator is
$$
C_u (I(p_i^\ast))^{-u/2} (1-o(1)).
$$
Thus we conclude that
\begin{equation} \label{eq:rx}
r_{i,S_i(\mathbi{p}^\ast)}^{\ell_u^u} (\mathbi{Q}) \ge C_u (I(p_i^\ast))^{-u/2} (1-o(1)) 
\quad \text{for all~} u\ge 1.
\end{equation}

Now we are left to calculate the value of $I(p_i^\ast)$.
To this end, we introduce some notation.
For a given privatization scheme $\mathbi{Q}\in\cD_{\epsilon,E}$ with output size $L$, we write its output alphabet as $\cY=\{1,2,\dots,L\}$, and we use the shorthand notation 
\begin{equation}\label{eq:lq}
q_{jv}:=\mathbi{Q}(j|v)
\end{equation}
 for all $j\in[L]$ and $v\in[k]$.
For $j\in[L]$ and $y^n=(y^{(1)},y^{(2)},\dots,y^{(n)})\in\cY^n$, define $w_j(y^n):=\sum_{v=1}^n \mathbbm{1}[y^{(v)}=j]$ to be
the number of times that symbol $j$ appears in $y^n$. Let $\mathbb{P}(y^n;p_i)$ be the probability mass function of a random vector $Y^n$ formed of i.i.d. samples drawn according to the distribution $\mathbi{m}=\mathbi{p}\mathbi{Q}$, 
where the other components of $\mathbi{p}$ are calculated from $p_i$ according to \eqref{eq:cpi}. The random variables
$w_j(Y^n)$ follow the multinomial distribution, and $\mathbb{E}w_j(Y^n)=n\mathbi{m}(j),j\in[L].$ Therefore,
\begin{align*}
\log \mathbb{P}(y^n;p_i) & = \sum_{j=1}^L w_j(y^n) \log \Big(\sum_{v=1}^k p_v q_{jv} \Big) \\
& = \sum_{j=1}^L w_j(y^n) \log\Big(p_i q_{ji} + \sum_{v\neq i} \Big(p_v^\ast-  \frac{1}{k-1}(p_i-p_i^\ast) \Big) q_{jv} \Big),
\end{align*}
and the Fisher information of $p_i$ is
\begin{align*}
I(p_i) & = - \underset{Y^n\sim (\mathbi{p}\mathbi{Q})^n}{\mathbb{E}} \left[\frac{d^2}{d p_i^2} \log \mathbb{P}(y^n;p_i) \right] \\
& =\sum_{j=1}^L \frac{(q_{ji}-\frac{1}{k-1}\sum_{v\neq i} q_{jv})^2}{\Big(p_i q_{ji} + \sum_{v\neq i} \Big(p_v^\ast-  \frac{1}{k-1}(p_i-p_i^\ast) \Big) q_{jv} \Big)^2}
\underset{Y^n\sim (\mathbi{p}\mathbi{Q})^n}{\mathbb{E}} w_j(Y^n) \\
& =\sum_{j=1}^L \frac{(q_{ji}-\frac{1}{k-1}\sum_{v\neq i} q_{jv})^2}
{\Big(\sum_{v=1}^k p_v q_{jv} \Big)^2}
\underset{Y^n\sim (\mathbi{p}\mathbi{Q})^n}{\mathbb{E}} w_j(Y^n) \\
& = n \sum_{j=1}^L \frac{(q_{ji}-\frac{1}{k-1}\sum_{v\neq i} q_{jv})^2}
{\sum_{v=1}^k p_v q_{jv} } \\
& = \frac{nk^2}{(k-1)^2} \sum_{j=1}^L \frac{(q_{ji}-\frac{1}{k}\sum_{v=1}^k q_{jv})^2}
{\sum_{v=1}^k p_v q_{jv} },
\end{align*}
where $p_v$'s on the last line are given by \eqref{eq:cpi}.
In particular,
$$
I(p_i^\ast) = \frac{nk^2}{(k-1)^2} \sum_{j=1}^L \frac{(q_{ji}-\frac{1}{k}\sum_{v=1}^k q_{jv})^2}
{\sum_{v=1}^k p_v^\ast q_{jv} }.
$$
Combining this with \eqref{eq:rx}, we have
$$
r_{i,S_i(\mathbi{p}^\ast)}^{\ell_u^u} (\mathbi{Q}) \ge 
C_u \Big( \frac{nk^2}{(k-1)^2} \sum_{j=1}^L \frac{(q_{ji}-\frac{1}{k}\sum_{v=1}^k q_{jv})^2}
{\sum_{v=1}^k p_v^\ast q_{jv} } \Big)^{-u/2} -o(n^{-u/2})
\quad \text{for all~} u\ge 1.
$$
For $j\in[L]$, define
\begin{equation}\label{eq:he}
q_j:=\frac{1}{k}\sum_{v=1}^k q_{jv}.
\end{equation}
It is clear that when $\mathbi{p}^\ast$ is in the neighborhood of the uniform distribution $\mathbi{p}_U$, i.e.,
when $p_v^\ast=1/k+o_n(1)$ for all $v\in[k]$, we have
\begin{equation}\label{eq:ks}
r_{i,S_i(\mathbi{p}^\ast)}^{\ell_u^u} (\mathbi{Q}) \ge 
C_u \Big( \frac{nk^2}{(k-1)^2} \sum_{j=1}^L \frac{(q_{ji}- q_j )^2}
{q_j } \Big)^{-u/2} -o(n^{-u/2})
\quad \text{for all~} u\ge 1.
\end{equation}

\subsection{Proof of \eqref{eq:mh}}
Our first step in this section will be to prove a lower bound on $r_{i,\Bayes}^{\ell_u^u} (\mathbi{Q})$.
Let us phrase the claim in \eqref{eq:ks} in a more detailed form: 
For any $\delta>0$, there exists $D_0>0$ such that whenever the constant $D'$ in the definition of $S_i(\mathbi{p}^\ast)$ is larger than $D_0$,
\begin{equation}\label{eq:ks1}
r_{i,S_i(\mathbi{p}^\ast)}^{\ell_u^u} (\mathbi{Q}) \ge 
C_u \Big( \frac{nk^2}{(k-1)^2} \sum_{j=1}^L \frac{(q_{ji}- q_j )^2}
{q_j } \Big)^{-u/2} -\delta n^{-u/2}
\quad \text{for all~} u\ge 1.
\end{equation}
The constant $D'$ is required to be large for the local asymptotic normality arguments to hold (refer again to \cite[Chapter 2, Theorem 1.1]{Ibrag81} and \cite[Chapter 6]{LCYang12}).

\begin{proposition} Let $\cP$ be the Euclidean ball around $\mathbi{p}_U$ defined in \eqref{eq:npu}.
For a sufficiently large constant $D$ and any $u\ge 1$ we have 
\begin{equation}\label{eq:elm}
r_{i,\Bayes}^{\ell_u^u} (\mathbi{Q}) \ge 
C_u \Big( \frac{nk^2}{(k-1)^2} \sum_{j=1}^L \frac{(q_{ji}- q_j )^2}
{q_j } \Big)^{-u/2} -o(n^{-u/2}).
\end{equation}
\end{proposition}
\begin{proof}
We can view $\cP$ as a union of (uncountably many) parallel line segments with direction vector $\mathbi{v}_i$ defined in \eqref{eq:vv}. 
Each of these line segments can be written as $S_i(\mathbi{p}^\ast)$ (see \eqref{eq:rgh}), with a suitably chosen
midpoint $\mathbi{p}^\ast\in \cP.$ Since the midpoints of all the line segments lie inside $\cP$, which is a neighborhood of the uniform distribution, by \eqref{eq:ks1} we have that
for any estimator $\hat{p}_i$, the average $\ell_u^u$ estimation loss 
$r_{i,S_i(\mathbi{p}^\ast)}^{\ell_u^u} (\mathbi{Q}, \hat{p}_i)$ on any of these line segments $S_i(\mathbi{p}^\ast)$ with $D'\ge D_0$ is lower bounded by 
$$
r_{i,S_i(\mathbi{p}^\ast)}^{\ell_u^u} (\mathbi{Q}, \hat{p}_i) \ge 
r_{i,S_i(\mathbi{p}^\ast)}^{\ell_u^u} (\mathbi{Q}) \ge
C_u \Big( \frac{nk^2}{(k-1)^2} \sum_{j=1}^L \frac{(q_{ji}- q_j )^2}
{q_j } \Big)^{-u/2} - \delta n^{-u/2}
$$
for $u\ge 1$. To compute the average estimation loss $r_{i,\Bayes}^{\ell_u^u} (\mathbi{Q}, \hat{p}_i)$ on $\cP$
we need to average over all the segments with weight proportional to the length of the segment. Given $D_0,$ we can choose $D$ in 
\eqref{eq:npu} large enough so that the proportion of the segments $S_i(\mathbi{p}^\ast)$ with $D'\ge D_0$ out of all the segments
in $\cP$ is arbitrarily close to one (formally, denote the union of such segments as $\cP_0$, then $\Vol(\cP_0)/\Vol(\cP)$ can be made arbitrarily close to $1$ as long as we set $D/D_0$ to be large enough). The average estimation loss along each of these segments is uniformly bounded below as in \eqref{eq:ks1}, and thus the average loss on $\cP_0$ is lower bounded by the same quantity. Combining the fact that $\Vol(\cP_0)/\Vol(\cP)=1-o(1)$, we have
$$
r_{i,\Bayes}^{\ell_u^u} (\mathbi{Q}, \hat{p}_i) \ge 
C_u \Big( \frac{nk^2}{(k-1)^2} \sum_{j=1}^L \frac{(q_{ji}- q_j )^2}
{q_j } \Big)^{-u/2} -o(n^{-u/2})
\quad \text{for all~} u\ge 1.
$$
This lower bound holds for any estimator $\hat{p}_i$, and this implies the claimed lower bound \eqref{eq:elm}.
\end{proof}

We will need the following lemma.\begin{lemma} \label{lem:ax}
For every $\mathbi{Q}\in\cD_{\epsilon,E}$ with output alphabet $\cY=\{1,2,\dots,L\}$ we have
$$
\sum_{i=1}^k \frac{q_{ji}^2}{q_j^2} \le k \Big( 1 +   (e^{\epsilon} - 1)^2  
 \frac{d^\ast (k-d^\ast)} {(d^\ast e^{\epsilon} + k - d^\ast)^2}  \Big) \quad\quad
\text{for all~} j\in[L].
$$
\end{lemma}
\begin{proof}
Let $m_j:=\min_{i\in[k]}q_{ji}.$
According to the definition of $\cD_{\epsilon,E}$ in \eqref{eq:DES}, 
the coordinates of the vector $(q_{ji},i\in[k])$ are either $m_je^\epsilon$ or $m_j.$
Let $d$ be the number of $m_je^\epsilon$ entries, then
\begin{align*}
q_j= \frac{m_j}{k}(d e^\epsilon + k-d), \\
\sum_{i=1}^k q_{ji}^2 = m_j^2(d e^{2\epsilon} + k-d) .
\end{align*}
We obtain
\begin{align*}
\sum_{i=1}^k \frac{q_{ji}^2}{q_j^2} & = 
\frac{k^2 (d e^{2\epsilon} + k-d)}{(d e^\epsilon + k-d)^2}
=k \frac{(d e^{2\epsilon} + k-d)(d+k-d)}{(d e^\epsilon + k-d)^2} \\
& = k \frac{d^2 e^{2\epsilon} + (k-d)^2 + d(k-d)(e^{2\epsilon}+1)}{(d e^\epsilon + k-d)^2} \\
& = k \frac{d^2 e^{2\epsilon} +2d(k-d)e^\epsilon + (k-d)^2 + d(k-d)(e^{2\epsilon} - 2e^\epsilon +1)}
{(d e^\epsilon + k-d)^2} \\
& = k \frac{(d e^\epsilon + k-d)^2 + d(k-d)(e^\epsilon -1)^2}
{(d e^\epsilon + k-d)^2}  \\
& = k\Big(1+(e^\epsilon -1)^2 \frac{d(k-d)}
{(d e^\epsilon + k-d)^2}\Big) \\
& \le k \Big( 1 +   (e^{\epsilon} - 1)^2  
 \frac{d^\ast (k-d^\ast)} {(d^\ast e^{\epsilon} + k - d^\ast)^2}  \Big),
\end{align*}
where the last inequality follows from the definition of $d^\ast$ in \eqref{eq:dstar}.
\end{proof}

Now we are ready to prove \eqref{eq:mh}. Using the obvious relations
$\sum_{j=1}^L q_{ji}=\sum_{j=1}^L q_j=1$,  we can simplify the right-hand side of \eqref{eq:elm}  as follows:
  \begin{align*}
    \sum_{j=1}^L \Big(\frac{(q_{ji}- q_j )^2}
{q_j } \Big)&=  \sum_{j=1}^L\Big(\sum_{j=1}^L \frac{q_{ji}^2}{q_j} - 2\sum_{j=1}^L q_{ji} + \sum_{j=1}^L q_j \Big) \Big)\\
  &=\sum_{j=1}^L \frac{q_{ji}^2}{q_j} - 1.
\end{align*}
Now let us sum \eqref{eq:elm} over $i\in[k]$ on both sides and use the simplification above:
\begin{equation} \label{eq:M}
\sum_{i=1}^k r_{i,\Bayes}^{\ell_u^u} (\mathbi{Q}) \ge 
C_u \sum_{i=1}^k \Big(\frac{nk^2}{(k-1)^2}\Big(\sum_{j=1}^L \frac{q_{ji}^2}{q_j} - 1 \Big) \Big)^{-u/2} -o(n^{-u/2}).
\end{equation}
Since for $u>0,$ $x^{-u/2}$ is a convex function for $x>0,$ we can further bound below the right-hand side of \eqref{eq:M}:   
\begin{align*}
   \sum_{i=1}^k &\Big(\frac{nk^2}{(k-1)^2}\Big(\sum_{j=1}^L \frac{q_{ji}^2}{q_j} - 1 \Big) \Big)^{-u/2}
  \ge k\Big( \frac{1}{k} \sum_{i=1}^k \frac{nk^2}{(k-1)^2} 
\Big(\sum_{j=1}^L \frac{q_{ji}^2}{q_j} - 1 \Big) \Big)^{-u/2}  \\
  &=k  \Big( \frac{nk}{(k-1)^2} \sum_{j=1}^L \sum_{i=1}^k \frac{q_{ji}^2}{q_j} - \frac{nk^2}{(k-1)^2} \Big)^{-u/2}  \\
  &= k \Big( \frac{nk}{(k-1)^2} \sum_{j=1}^L \Big( q_j\sum_{i=1}^k \frac{q_{ji}^2}{q_j^2} \Big) - \frac{nk^2}{(k-1)^2} \Big)^{-u/2}  \\
  &\ge k   
\Big( \frac{nk^2}{(k-1)^2} \Big( 1 +   (e^{\epsilon} - 1)^2  
 \frac{d^\ast (k-d^\ast)} {(d^\ast e^{\epsilon} + k - d^\ast)^2}  \Big) \sum_{j=1}^L  q_j - \frac{nk^2}{(k-1)^2} \Big)^{-u/2}  \\
  &=k \Big( \frac{nk^2 (e^{\epsilon} - 1)^2}{(k-1)^2}     
 \frac{d^\ast (k-d^\ast)} {(d^\ast e^{\epsilon} + k - d^\ast)^2}  \Big)^{-u/2} \\
 &=\frac{k}{n^{u/2}}M(k,\epsilon)^{u/2}
\quad\quad\quad \text{~~~for all~} \mathbi{Q}\in\cD_{\epsilon,E} 
   \end{align*}
 where the second inequality follows by Lemma \ref{lem:ax} (note the inverted inequality of the Lemma
 because of the negative power $-u/2$). Combining this with \eqref{eq:M}, we conclude that
   \begin{align*}
   \sum_{i=1}^k   r_{i,\Bayes}^{\ell_u^u} (\mathbi{Q}) \ge \frac{k}{n^{u/2}} C_u M(k,\epsilon)^{u/2} - o(n^{-u/2})  \quad \text{~~~for all~} \mathbi{Q}\in\cD_{\epsilon,E}.
   \end{align*}
 Thus we have established \eqref{eq:mh}, and this completes the proof of Theorem~\ref{Thm:Main}.  

\section{Proof of Theorem~\ref{thm:ac}}\label{Sect:pac}
We begin with showing that for the privatization scheme $\mathbi{Q}_{k,\epsilon,d}$ defined in \eqref{eq:defQ} and the estimator  \eqref{eq:emp}, the $\ell_u^u$ estimation loss is maximized for the uniform distribution $\mathbi{p}_U$ for all $0<u\le 2$ when $n$ is large. To shorten the notation, rewrite \eqref{eq:emp} as
$$
\hat{p_i}(y^n)= A \frac{t_i(y^n)}{n} - B, \quad i\in[k],
$$
where
$$
A:= \frac{(k-1)e^{\epsilon}+\frac{(k-1)(k-d)}{d}}{(k-d)(e^{\epsilon}-1)}, \quad \quad
B:= \frac{(d-1)e^{\epsilon}+k-d}{(k-d)(e^{\epsilon}-1)}.
$$
In \cite{Ye17} we have shown that the estimator $\hat{p_i}(y^n)$ is unbiased, i.e.,
$$
p_i = A \underset{Y^n\sim (\mathbi{p}\mathbi{Q}_{k,\epsilon,d})^n}{\mathbb{E}}
\Big(\frac{t_i(Y^n)}{n}\Big) - B, \quad  i\in[k].
$$
By definition,
$$
t_i(Y^n)=\sum_{j=1}^n \mathbbm{1}[Y_i^{(j)}=1]
$$
is the sum of $n$ i.i.d. Bernoulli random variables with parameter
$$
\mathbb{P}[Y_i^{(j)}=1]=\mathbb{E} \frac{t_i(Y^n)}{n} = \frac{p_i}{A} + \frac{B}{A}.
$$
Therefore the variance of $\frac{t_i(Y^n)}{n}$ is $\frac{1}{n}(\frac{p_i}{A} + \frac{B}{A})
(1-\frac{p_i}{A} - \frac{B}{A})$, and the variance of $\hat{p_i}(Y^n)$ is
  \begin{gather*}
\text{Var}\; \hat{p_i}(Y^n)=A^2 \frac{1}{n}\Big(\frac{p_i}{A} + \frac{B}{A}\Big)
\Big(1-\frac{p_i}{A} - \frac{B}{A}\Big)
= \frac{1}{n}(p_i+B)(A-p_i-B).
  \end{gather*}
  Using the Central Limit Theorem, we then obtain for the absolute moment of $\hat{p_i}(Y^n)$ around $p_i$ the following
  approximation:
$$
\underset{Y^n\sim (\mathbi{p}\mathbi{Q}_{k,\epsilon,d})^n}{\mathbb{E}}
|\hat{p_i}(Y^n)-p_i|^u
= C_u \Big( \frac{1}{n}(p_i+B)(A-p_i-B) \Big)^{u/2} + o(n^{-u/2}),
$$
where $C_u$ is the absolute moment of the $\cN(0,1)$ RV; see Section~\ref{Sect:ovr}.
Therefore,
\begin{align*}
\underset{Y^n\sim (\mathbi{p}\mathbi{Q}_{k,\epsilon,d})^n}{\mathbb{E}} &
\ell_u^u(\hat{\mathbi{p}}(Y^n),\mathbi{p})
 = \sum_{i=1}^k C_u \Big( \frac{1}{n}(p_i+B)(A-p_i-B) \Big)^{u/2} + o(n^{-u/2}) \\
& \le k C_u n^{-u/2} \Big(\frac{1}{k} \sum_{i=1}^k (p_i+B)(A-p_i-B) \Big)^{u/2} + o(n^{-u/2}) \\
& = k C_u n^{-u/2} \Big( \frac{A}{k}-\frac{2B}{k}+AB-B^2-\frac{1}{k}\sum_{i=1}^k p_i^2 \Big)^{u/2} + o(n^{-u/2}) \\
& \le  k C_u n^{-u/2} \Big( \frac{A}{k}-\frac{2B}{k}+AB-B^2-\frac{1}{k^2} \Big)^{u/2} + o(n^{-u/2}),
\end{align*}
where the first inequality follows from the fact that $x^{u/2}$ is a concave function of $x$ on $(0,+\infty)$ for all positive $0<u\le 2$, and the last line uses the Cauchy--Schwarz inequality. Both inequalities hold with equality if and only if $\mathbi{p}$ is the uniform distribution. Thus when $n$ is large, for all $0<u\le 2$ and all $1\le d\le k-1$, we have
$$
r_{k,n}^{\ell_u^u} (\mathbi{Q}_{k,\epsilon,d}, \hat{\mathbi{p}})  =
\underset{Y^n\sim (\mathbi{p}_U\mathbi{Q}_{k,\epsilon,d})^n}{\mathbb{E}} \ell_u^u(\hat{\mathbi{p}}(Y^n),\mathbi{p}_U).
$$
In particular, it also holds for $d=d^\ast$.
Next we calculate the estimation loss at the uniform distribution. By symmetry, it is clear that
\begin{align*}
\underset{Y^n\sim (\mathbi{p}_U\mathbi{Q}_{k,\epsilon,d^\ast})^n}{\mathbb{E}}
\Big|\hat{p_i}(Y^n)-\frac{1}{k} \Big|^2
& = \frac{1}{k} \Big( \underset{Y^n\sim (\mathbi{p}_U\mathbi{Q}_{k,\epsilon,d^\ast})^n}{\mathbb{E}} \ell_2^2(\hat{\mathbi{p}}(Y^n),\mathbi{p}_U)  \Big) \\
& = \frac{1}{k}  r_{k,n}^{\ell_2^2} (\mathbi{Q}_{k,\epsilon,d^\ast}, \hat{\mathbi{p}})
= \frac{M(k,\epsilon)}{n}.
\end{align*}
Therefore when the input distribution is uniform, $\hat{p_i}(Y^n)$ can be approximated for large $n$ by a Gaussian random variable with mean $1/k$ and variance $\frac{M(k,\epsilon)}{n}.$ 
Thus,
$$
\underset{Y^n\sim (\mathbi{p}_U\mathbi{Q}_{k,\epsilon,d^\ast})^n}{\mathbb{E}}
\Big|\hat{p_i}(Y^n)-\frac{1}{k} \Big|^u
= C_u \Big( \frac{M(k,\epsilon)}{n} \Big)^{u/2} + o(n^{-u/2}),
$$
so for $0<u\le 2$,
\begin{align*}
r_{k,n}^{\ell_u^u} (\mathbi{Q}_{k,\epsilon,d^\ast}, \hat{\mathbi{p}}) & = 
\underset{Y^n\sim (\mathbi{p}_U\mathbi{Q}_{k,\epsilon,d^\ast})^n}{\mathbb{E}} \ell_u^u(\hat{\mathbi{p}}(Y^n),\mathbi{p}_U) \\
& = \frac{k}{n^{u/2}} C_u M(k,\epsilon)^{u/2} + o(n^{-u/2}).
\end{align*}
This completes the proof of Theorem~\ref{thm:ac}.

\bibliographystyle{IEEEtran}
\bibliography{differential}

\end{document}